\documentclass[11pt,reqno]{amsart}
\usepackage{indentfirst, amssymb, amsmath, amsthm, mathrsfs, setspace, indentfirst, enumerate,  mathrsfs, amsmath, amsthm}
\usepackage[bookmarksnumbered, colorlinks, plainpages]{hyperref}
\usepackage{mathrsfs}
\usepackage{cite}
\usepackage{tikz}
\usepackage{graphicx}
\usepackage{placeins}
\usetikzlibrary{arrows.meta,calc,decorations.pathreplacing,positioning,shadows.blur,backgrounds,calc,fit,
	decorations.pathreplacing}
\usetikzlibrary{shapes.geometric}
\usepackage{float}
\usepackage{pgfplots}
\usepackage{booktabs}
\usepackage{threeparttable} 
\usetikzlibrary{shadows}
\pgfplotsset{compat=1.18}
\usepackage[table]{xcolor}
\usepackage{array}
\usepackage{makecell}
\pgfplotsset{compat=1.18}
\usepackage{smartdiagram}
\usesmartdiagramlibrary{additions}
\definecolor{myblue}{RGB}{34,72,140}
\definecolor{lightblue}{RGB}{237,243,252}
\definecolor{lightgreen}{RGB}{239,248,239}
\definecolor{lightorange}{RGB}{255,247,235}
\definecolor{lightred}{RGB}{252,241,241}
\definecolor{HeaderBlue}{RGB}{40,76,80}
\definecolor{LightBlue}{RGB}{235,242,250}
\definecolor{AltRow}{RGB}{248,250,253}
\usepgfplotslibrary{fillbetween}
\definecolor{curveblue}{RGB}{30,90,220}
\definecolor{negred}{RGB}{255,182,193}
\definecolor{posgreen}{RGB}{40,150,70}
\definecolor{rootdark}{RGB}{120,40,20}
\usepgfplotslibrary{fillbetween}
\definecolor{c1}{RGB}{70,40,200}
\definecolor{c2}{RGB}{213,14,0}
\definecolor{c3}{RGB}{0,158,10}
\definecolor{c4}{RGB}{230,126,34}
\newtheorem*{theoA}{Theorem A}
\newtheorem*{theoB}{Theorem B}
\newtheorem*{theoC}{Theorem C}

\newtheorem*{cor A}{Corollary A}
\newtheorem*{cor B}{Corollary B}

\newtheorem{theo}{Theorem}[section]
\newtheorem{lem}{Lemma}[section]
\newtheorem{cor}{Corollary}[section]

\newtheorem{defi}{Definition}[section]
\newtheorem{rem}{Remark}[section]
\newcommand{\ol}{\overline}
\newcommand{\be}{\begin{equation}}
	\newcommand{\ee}{\end{equation}}
\newcommand{\beas}{\begin{eqnarray*}}
	\newcommand{\eeas}{\end{eqnarray*}}
\newcommand{\bea}{\begin{eqnarray}}
	\newcommand{\eea}{\end{eqnarray}}

\numberwithin{equation}{section}

\begin{document}
\title[T\MakeLowercase {he} F\MakeLowercase{ekete}-S\MakeLowercase{zeg\"{o} problem for}......]{\LARGE T\Large\MakeLowercase {he} F\MakeLowercase{ekete}-S\MakeLowercase{zeg\"{o} problem for starlike mappings of order $\alpha$ and quasi-convex mappings of Type $B$ in the unit balls of complex} B\MakeLowercase{anach spaces}}

\date{}
\author[S. M\MakeLowercase{ajumder}, A. B\MakeLowercase{anerjee}, G. H\MakeLowercase{aldar} \MakeLowercase{and} S. P\MakeLowercase{anja} ]{S\MakeLowercase{ujoy} M\MakeLowercase{ajumder}$^{1}$, A\MakeLowercase{bhijit} B\MakeLowercase{anerjee}$^2$,  G\MakeLowercase{outam} H\MakeLowercase{aldar}$^3$$^*$  \MakeLowercase{and} S\MakeLowercase{hantanu} P\MakeLowercase{anja}$^4$}

\address{$^{1}$ Sujoy Majumder, Department of Mathematics, Raiganj University, Raiganj, West Bengal-733134, India.}
\email{sm05math@gmail.com, sjm@raiganjuniversity.ac.in}
\address{$^{2}$ Abhijit Banerjee, Department of Mathematics, University of Kalyani, West Bengal 741235, India.}
\email{abanerjee\_kal@yahoo.co.in, abhijitbanerjee@klyuniv.ac.in}
\address{$^3$ Goutam Haldar, Department of Mathematics, Ghani Khan Choudhury Institute of Engineering and Technology, Narayanpur, Malda-732141, West Bengal, India.}
\email{goutamiitm@gmail.com, goutamiit1986@gmail.com}
\address{$^4$ Shantanu Panja, Department of Mathematics, University of Kalyani, West Bengal 741235, India.}
\email{panjasantu07@gmail.com, shantanumath26@klyuniv.ac.in}

\renewcommand{\thefootnote}{}
\footnote{2020 \emph{Mathematics Subject Classification}: 32H02, 30C45.}
\footnote{\emph{Key words and phrases}: Biholomorphic mapping, Starlike mapping, quasi-convex mapping, Coefficient problems, Zalcman conjecture, Fekete-Szeg\"{o} problem
.}
\footnote{*\emph{Corresponding Author}: Goutam Halder.}

\renewcommand{\thefootnote}{\arabic{footnote}}
\setcounter{footnote}{0}

\begin{abstract}
	 We establish sharp estimates for the generalized Zalcman functional corresponding to the index $n=2$ for normalized starlike mappings of order $\alpha$ and normalized quasi-convex mappings of type $B$ in the unit ball of a complex Banach space and on the unit polydisk in $\mathbb{C}^n$. We also derive sharp Fekete–Szeg\"o inequalities for normalized starlike mappings of order $\alpha$ and normalized quasi-convex mappings of type $B$ in the unit ball of a complex Banach space. Our results extend classical coefficient estimates to higher-dimensional holomorphic mappings and provide a partial confirmation of the generalized Zalcman conjecture in higher dimensions.

\end{abstract}
\thanks{Typeset by \AmS -\LaTeX}
\maketitle
\section{\bf Introduction}
Geometric function theory is a fundamental branch of complex analysis that studies how the analytic properties of holomorphic mappings determine their geometric behaviour. At the heart of this theory lies the class of normalized univalent functions, which provides a natural connection between conformal mapping, geometric inequalities, extremal problems and coefficient theory. The Taylor coefficients of these functions carry subtle geometric information and satisfy sharp estimates for various nonlinear coefficient functionals. This has become one of the central themes of the subject. This line of research has led to several landmark problems, including the Bieberbach conjecture, the Fekete--Szeg\"{o} problem and the Zalcman conjecture and continues to inspire investigations into increasingly broad subclasses of univalent functions.
\par Denote by $\mathcal{S}(\mathbb{U})$ the collection of functions $f$ univalent on the unit disk $\mathbb{U}=\{z\in\mathbb{C}:|z|<1\}$, subject to the 
normalization $f(0)=0$ and $f'(0)=1$. Every such $f$ admits the expansion
\begin{align}\label{eq1}
	f(z)=z+\sum_{n=2}^{\infty}a_nz^n,\qquad z\in\mathbb{U}.
\end{align}

In the late 1960s, L. Zalcman posed the conjecture that if $f \in \mathcal{S}(\mathbb{U})$ is given by the power series \eqref{eq1}, then its Taylor coefficients satisfy the inequality
\begin{align}\label{eq2}
\displaystyle	
\vert{}a_n^2 - a_{2n-1}\vert{} \leqslant (n - 1)^2, \quad \text{for all } n \geqslant 2,
\end{align}
with equality holding if and only if $f$ is the Koebe function or one of its rotations.

Brown and Tsao \cite{Brown-Tsao-1986} first pointed out that the Zalcman conjecture implies the Bieberbach conjecture.
A well-known consequence of the area theorem shows that inequality \eqref{eq2} holds for $n = 2$. For $f \in \mathcal{S}(\mathbb{U})$, Krushkal \cite{Krushkal-1995,Krushkal-2010} proved the conjecture for $n = 3$, and subsequently extended the result to $n = 4, 5,$ and $6$. For a simple and elegant proof of the case $n = 3$, we refer the reader to \cite{Krushkal-2010}. For $n \geqslant 7$, the conjecture for $f \in \mathcal{S}(\mathbb{U})$ remains open.\vspace{1.2mm}

A normalized univalent function $f \in \mathcal{S}(\mathbb{U})$ is said to be starlike (with respect to the origin) if and only if
\begin{align*}
\displaystyle \Re\bigg\{ \frac{z f'(z)}{f(z)} \bigg\}> 0,\qquad z \in \mathbb{U}.
\end{align*}
 The subclass consisting of all such starlike functions is denoted by $\mathcal{S}^*(\mathbb{U})$. The Zalcman conjecture has been studied for several subclasses of $\mathcal{S}(\mathbb{U})$. Brown and Tsao \cite{Brown-Tsao-1986} settled it for both starlike and typically real functions, while Ma \cite{Ma-1988} later proved that it holds for close-to-convex functions as well.

The classical Fekete–Szeg\"{o} inequality for the coefficients of $f \in \mathcal{S}(\mathbb{U})$, proved via Loewner’s method, is given by (see \cite{Fekete-Szego-1933}; e.g.,
\cite[p. 124]{Pommerenke-1975}):
\begin{align*}
\displaystyle
\vert{}a_3 - \lambda a_2^2\vert{} \le 1 + 2 \exp\left(-2\lambda/(1 - \lambda)\right), \quad \text{for } \lambda \in [0, 1).
\end{align*}

Koepf \cite{Koepf-1987} obtained the following Fekete–Szegö inequality for normalized starlike functions on $\mathbb{U}$:
\begin{align*}
\displaystyle
\vert{}a_3 - \lambda a_2^2\vert{} \le \max\{1, \vert{}3 - 4\lambda\vert{}\}, \quad \text{for } \lambda \in \mathbb{C}.
\end{align*}
Later, Kanas \cite{Kanas-2012} provided a unified approach to the Fekete–Szegö problem on $\mathbb{U}$.

Let $f \in \mathcal{S}(\mathbb{U})$ and $0 \le \alpha < 1$. Then $f \in \mathcal{S}^*_{\alpha}(\mathbb{U})$ if and only if \begin{align*}
\displaystyle \Re\bigg\{ \frac{z f'(z)}{f(z)} \bigg\} > \alpha,\quad z \in \mathbb{U}.
\end{align*}
 In 1969, Keogh and Merkes \cite{Keogh-Merkes-1969} obtained sharp estimates for the Fekete--Szeg\"{o} problem concerning
 starlike functions of order $\alpha$.
\begin{theoA}Let $\alpha\in[0,1)$. If $f \in \mathcal{S}^*_{\alpha}(\mathbb{U})$ is of the form \eqref{eq1} and $\mu$ be any complex number, then 
\begin{align*}
\displaystyle
|a_3-\mu a_2^2|\leq (1-\alpha)\max\{1,|2(1-\alpha)(2\mu-1)-1|\}.
\end{align*}
The inequality is sharp.	
\end{theoA}

Let $\mathcal{K}(\mathbb{U})$ denote the familiar subclass of $\mathcal{S}(\mathbb{U})$ consisting of functions which are convex. Then $f\in\mathcal{K}(\mathbb{U})$ if and only if 
\begin{align*}
\displaystyle \Re\bigg\{ 1+\frac{z f''(z)}{f'(z)} \bigg\} > 0,\quad z \in \mathbb{U}.
\end{align*}

For normalized convex functions on $\mathbb{U}$, Keogh and Merkes \cite{Keogh-Merkes-1969} obtained the following result:
\begin{theoB} Let $f\in\mathcal{K}(\mathbb{U})$. Then for $\mu\in\mathbb{C}$, we have the following sharp inequality
\begin{align*}
\displaystyle
	|a_3-\mu a_2^2|\leq \frac{1}{3}\max\{1,3|1-\mu|\}.
\end{align*}	
\end{theoB}

The following result, first obtained by Hummel \cite{Hummel-1957} and later by Trimble \cite{Trimble-1975}, has remarkable implications in the theory of univalent functions.
\begin{theoC} Let $f\in\mathcal{K}(\mathbb{U})$. Then
\begin{align*}
|a_3-a_2^2|\leq \frac{1}{3}\left(1-|a_2|^2\right).
\end{align*}	
\end{theoC}

Let $X$ be a complex Banach space with norm $||.||$. Let $\mathbb{B}_r=\{z\in X:\|z\|<r\}$ denote the ball of radius $r$. If $r=1$, we will simply use $\mathbb{B}$ and if $X=\mathbb{C}$, then $\mathbb{B}=\mathbb{U}$. Let $\partial \mathbb{B}$ denote the boundary of $\mathbb{B}$, $\overline{\mathbb{B}}$ its closure, and $\partial_0 \mathbb{U}^n$ the characteristic boundary (i.e., the boundary on which the maximum modulus of a holomorphic function is attained) of the unit polydisk $\mathbb{U}^n$.\vspace{1.2mm}

Now let $X$ and $Y$ be two complex Banach spaces with respect to the norms $||.||_1$ and $||.||_2$ respectively. For simplicity, we denote both norms by $||.||$, when there is no possibility of confusion. Let $L(X, Y)$ denote the set of continuous linear operators from a complex Banach space $X$ into a complex Banach space $Y$ with the standard operator norm 
\begin{align*} 
	||A||=\sup\{||A(z)||: ||z||=1,\quad A\in L(X,Y)\}.
\end{align*} 

Let $I$ be the identity operator in $L(X, X)$. For each $z \in X \setminus \{0\}$, define the set  
\begin{align*}
	\displaystyle 
	T(z) = \{ l_z \in L(X, \mathbb{C}) : l_z(z) = \Vert{}z\Vert{}, \quad \Vert{}l_z\Vert{} = 1 \}.
\end{align*}

This set is guaranteed to be nonempty by the Hahn-Banach theorem. 
Let $\Omega$ be a domain in $X$ and let $f:\Omega \rightarrow Y$ be a mapping. Then $f$ is called holomorphic on $\Omega$ if for any $z\in\Omega$ there is a mapping $Df(z)\in L(X,Y)$, called the Fr\'echet derivative of $f$ at $z$ such that 
\begin{align*} \lim_{h\to 0} \frac{||f(z+h)-f(z)-Df(z)h||}{||h||}=0.\end{align*} 

Let $\mathcal{H}(\Omega, \Omega')$ represent the set of all holomorphic mappings from a domain $\Omega \subset X$ into another domain $\Omega' \subset Y$. We simply write $\mathcal{H}(\Omega, X) = \mathcal{H}(\Omega)$. It is well known that if $f \in \mathcal{H}(\Omega,Y)$, then for all $w$ in some neighborhood of $z \in \Omega$, the function can be expressed as the power series expansion:
\begin{align*}
	\displaystyle
	f(w) = \sum_{n=0}^{\infty} \frac{1}{n!} D^n f(z)((w - z)^n),
\end{align*}
where $D^n f(z)$ is a bounded symmetric $n$-linear mapping from $\prod_{j=1}^n X$ into $Y$, called the $n$-th Fr\'echet derivative of $f$ at $z$. Moreover, the term $D^n f(z)((w - z)^n)$ is shorthand for applying the $n$-linear mapping to $n$ copies of the vector $(w - z)$:
\begin{align*}
	\displaystyle
	D^n f(z)((w - z)^n) = D^n f(z)(\underbrace{w - z, w - z, \dots, w - z}_{n \text{ times}})
\end{align*}
where $D^0f(z)=f(z)$.
A mapping $f \in \mathcal{H}(\mathbb{B})$ is said to be normalized if $f(0) = 0$ and $Df(0) = I$, where $I \in L(X,X)$ is the identity operator. It is called biholomorphic if $f$ is bijective with a holomorphic inverse $f^{-1}$ on $f(\mathbb{B})$, and locally biholomorphic if $Df(z)$ has a bounded inverse for every $z \in \mathbb{B}$. The family of all normalized biholomorphic mappings on $\mathbb{B}$ is denoted by $\mathcal{S}(\mathbb{B})$. A biholomorphic mapping on $\mathbb B$ is said to be starlike
if $f(\mathbb B)$ is a starlike domain with respect to the origin.\vspace{1.2mm}


It is known that if $X=\mathbb{C}^{n}$ with respect to maximum norm $\|.\|_{\infty}$ and $l_z\in T(z)$, then
\begin{align}\label{eq3}
\displaystyle
	l_z(w)=
	\sum_{\{j:\|z\|_{\infty}=|z_j|\}}
	t_j\,
	\frac{w_j\overline{z_j}}
	{\|z\|_{\infty}},
	\qquad \text{for }p=\infty,
\end{align}
where each $t_j\geq0$ and $\sum\limits_{\{j:\|z\|_{\infty}=|z_j|\}} t_j=1.$

The following family plays a key role in our discussion:
\begin{align*}\displaystyle
	\mathcal{M}=\{f:\mathbb{B}\to X: f\in\mathcal{H}(\mathbb{B}),\; f(0)=0,\;Df(0)=I,\; \Re\{l_z(f(z))\}> 0,\; z\in\mathbb{B}\backslash\{0\},\; 
\end{align*}

In the case of the maximum norm $\Vert{}.\Vert{}_\infty $ on $\mathbb{C}^n$, the class $\mathcal{M}$ consists of all normalized holomorphic mappings $w:\mathbb{U}^n \to \mathbb{C}^n$ satisfying,
\begin{align*}
\displaystyle
\Re\left\lbrace \frac{w_j(z)}{z_j} \right\rbrace >0,
\end{align*}
whenever $\Vert{}z\Vert{}_\infty = \vert{}z_j\vert{} > 0$ for $1 \le j \le n$.
\FloatBarrier
\begin{table}[H]
	\centering
	\caption{Summary of notation used in the Banach space setting.}
	\label{tab:notation}
	\begin{tabular}{ll}
		\toprule
		Symbol & Meaning \\
		\midrule
		$X$, $\|\cdot\|$ & Complex Banach space and its norm \\
		$\mathbb{B}$, $\mathbb{B}_r$ & Unit ball, ball of radius $r$ in $X$ \\
		$T(z)$ & Set of norming functionals at $z\in X\setminus\{0\}$ \\
		$\mathcal{H}(\Omega,\Omega')$ & Holomorphic mappings $\Omega\to\Omega'$ \\
		$D^nf(z)$ & $n$-th Fr\'echet derivative of $f$ at $z$ \\
		$\mathbb{B}(p)$ & Unit ball of $\mathbb{C}^n$ w.r.t.\ the $p$-norm \\
		$\mathbb{U}^n$ & Unit polydisk ($=\mathbb{B}(\infty)$) \\
		\bottomrule
	\end{tabular}
\end{table}

In several complex variables, classical results often fail to hold in their original form. Most notably, Poincar\'e disproved the higher-dimensional analogue of the Riemann mapping theorem, while Cartan constructed counterexamples to the Bieberbach conjecture on the unit ball. As a result, extending univalent function theory to higher dimensions requires introducing appropriate subclasses or additional structural restrictions.
Recently bounds for coefficients of various subclasses of normalized biholomorphic
mappings in higher dimensions have been obtained by Bracci \cite{Bracci-2015}, Bracci et al. \cite{Bracci-2016},
Bracci and Roth \cite{Bracci-2018}, Graham et al. \cite{Graham-2002}, \cite{Graham-2016}, \cite{Graham-2017}, Hamada and Honda \cite{Hamada-Honda-2008}, Hamada et al. \cite{Hamada-Honda-Kohr-2006}, Hamada and Kohr \cite{Hamada-Kohr-2020} and Xu et al. \cite{Xu-Liu-Liu-2018}.

The discussion above naturally leads to the following summary, which places the present work in the context of the existing literature.
\begin{table}[H]
	\centering
	\begin{threeparttable}
		\caption{Evolution of coefficient problems from the classical one-variable setting to higher-dimensional holomorphic mappings.}
		\label{tab:comparison}
		
		\renewcommand{\arraystretch}{1.45}
		
		\rowcolors{2}{AltRow}{white}
		
		\begin{tabular}{
				>{\raggedright\arraybackslash}p{2.5cm}
				>{\raggedright\arraybackslash}p{3.4cm}
				>{\raggedright\arraybackslash}p{3.cm}
				>{\raggedright\arraybackslash}p{4.7cm}
			}
			
			\toprule
			
			\rowcolor{HeaderBlue}
			
			\color{white}\bfseries Problem
			&
			\color{white}\bfseries Classical Setting
			&
			\color{white}\bfseries Higher-Dimensional Status
			&
			\color{white}\bfseries Contribution of the Present Paper
			
			\\
			
			\midrule
			
			\textbf{Bieberbach theory}
			&
			Completely resolved by de Branges
			&
			Fails in general for biholomorphic mappings
			&
			Provides the motivating framework
			
			\\
			
			\textbf{Fekete--Szeg\"o problem}
			&
			Sharp inequalities known for numerous subclasses
			&
			Only partially developed
			&
			Sharp inequalities for starlike and convex mappings on Banach spaces
			
			\\
			
			\textbf{Zalcman conjecture}
			&
			Open for general univalent functions ($n\ge7$)
			&
			Very limited literature
			&
			Sharp estimates for the case $n=2$ in higher dimensions
			
			\\
			
			\textbf{Coefficient functionals}
			&
			Extensively studied in one variable
			&
			Relatively unexplored
			&
			Extends several classical coefficient estimates to biholomorphic mappings
			
			\\
			
			\bottomrule
			
		\end{tabular}

	\end{threeparttable}
\end{table}
\tikzset{
	box/.style={
		draw=myblue,
		rounded corners=4pt,
		very thick,
		fill=LightBlue,
		text width=14.3cm,
		inner sep=6pt,
		drop shadow
	},
	mainbox/.style={
		draw=myblue,
		rounded corners=4pt,
		very thick,
		fill=lightgreen,
		text width=14.3cm,
		inner sep=7pt,
		drop shadow
	}
}
\begin{center}
	\begin{tikzpicture}
		\node[box]{
			\textbf{\large Motivation and Research Gap}
		
		\vspace{1mm}
		
		Although the classical Fekete--Szeg\"o problem and the generalized	Zalcman conjecture have been extensively investigated for normalized
		univalent functions in one complex variable, their higher-dimensional	counterparts remain far less understood.
		In particular, sharp coefficient estimates for normalized starlike	mappings of order $\alpha$ and normalized convex mappings on the unit
		ball of a complex Banach space and the unit polydisk are still	relatively scarce, leaving several natural extremal problems open.
		};
	\end{tikzpicture}
\end{center}

\begin{center}
	\begin{tikzpicture}
		\node[mainbox]{
			\textbf{\large Main Contributions of the Present Paper}
		
		\vspace{1mm}
		
		1. We establish sharp Fekete--Szeg\"o inequalities for normalized	starlike mappings of order $\alpha$ and convex mappings	on the unit ball of a complex Banach space.\\
		2. We obtain sharp estimates for the generalized Zalcman	functional in the case $n=2$ for mappings defined on both the unit ball and the unit polydisk.\\
		3. Our results extend the classical one-variable
			Theorems~A--C to higher-dimensional holomorphic mappings. \\
		4. These estimates provide a partial confirmation of the generalized	Zalcman conjecture in higher dimensions.		
		
		};
	\end{tikzpicture}	
\end{center}
\section{{\bf Fekete-Szeg\"{o} inequalities for starlike mappings of order $\alpha$ on complex Banach space}}\label{Sec-2}
In this section, we first recall the definition of starlikeness of order $\alpha$ in complex Banach space.
\begin{defi}\emph{\cite[Definition 6.2.12]{Graham-Kohr}}
	Let $f:\mathbb{B}\to X$ be a locally biholomorphic mapping such that $f(0)=0$ and let $\alpha\in [0,1)$. We say that $f$ is starlike of order $\alpha$ if 
	\begin{align}\label{SL1}
	\displaystyle
		\Re\left\lbrace \frac{||z||}{l_z([Df(z)]^{-1}f(z))}\right\rbrace>\alpha.\quad z\in\mathbb{B}\backslash\{0\},\quad l_z\in T(z).
	\end{align}
\end{defi}

The definition of starlikeness of order $\alpha$ extends naturally to complex Banach spaces; see \cite[Definition 1.2]{Hamada-2001}.

It  is easy to see that in one variable the condition \eqref{SL1} reduces to $\displaystyle \Re\left\lbrace \frac{zf'(z)}{f(z)}\right\rbrace>\alpha$, that is the usual condition for starlikeness of order $\alpha$ in the unit disk $\mathbb{U}$.

Let $\mathbb{B}$ be the unit ball of a complex Banach space $X$. We denote by $\mathcal{S}_\alpha^*(\mathbb{B})$ the subclass of $\mathcal{S}(\mathbb{B})$ consisting of normalized starlike mappings of order $\alpha$ from $\mathbb{B}$ into $X$.
 
 The theory of starlike mappings in higher dimensions has been extensively developed. For more information, see the recent monographs by Gong  \cite{Gong} and Kohr--Liczberski  \cite{Kohr- Liczberski-1998}. Foundational papers on the subject include \cite{RCS_PJM_1991,Brown-Tsao-1986,Chu,SSQ_CHMS_1990}, \cite{Kub-Por}, \cite{Su2}-\cite{Xu-He-Xu}, \cite{Xu-Liu-2010}.
 \vspace{1.2mm}
 
 The following theorem provides a sharp bound for the Zalcman functional for the class of starlike mappings of order $\alpha$ in higher dimensions.
 
 \begin{theo}\label{Th-1} Let $f\in\mathcal{H}(\mathbb{B},\mathbb{C})$ satisfy $f(0)=1$,
 	and suppose that the mapping $F:\mathbb{B}\to X$ defined by 	$F(z)=f(z)z$ belongs to $\mathcal{S}_\alpha^*(\mathbb{B})$.
 	Then, for every $\mu\in\mathbb{C}$ and 
 	$z\in\mathbb{B}\setminus\{0\}$,
 \begin{align}\label{Th-A}
 \displaystyle
 		\left|\frac{l_{z}(D^3F(0)(z^3))}{3!\|z\|^3}-\mu\left(\frac{l_z(D^2F(0)(z^2))}{2!\|z\|^2}\right)^2\right|\leq 
 		(1-\alpha)\max\{1,|2(1-\alpha)(2\mu-1)-1|\},
 	\end{align}
 for $z\in\mathbb{B}\backslash\{0\}$. The inequality \eqref{Th-A} is sharp.
 \end{theo}

 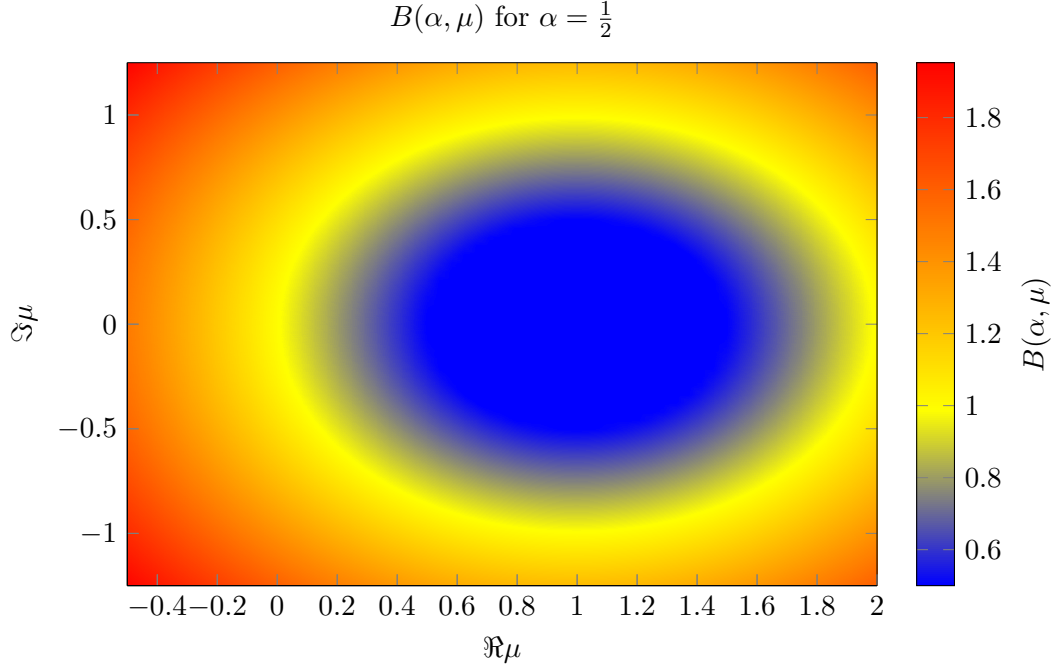
\begin{figure}[H]
 	\centering
 	\begin{tikzpicture}
 		\begin{axis}[
 			width=11.5cm,
 			height=8.5cm,
 			view={0}{90},
 			xlabel={$\Re\mu$},
 			ylabel={$\Im\mu$},
 			xmin=-0.5, xmax=2,
 			ymin=-1.25, ymax=1.25,
 			colorbar,
 			colorbar style={
 				ylabel={$B(\alpha,\mu)$}
 			},
 			title={$B(\alpha,\mu)$ for $\alpha=\tfrac12$},
 			samples=61,
 			samples y=61,
 			domain=-0.5:2,
 			y domain=-1.25:1.25,
 			]
 			\addplot3[
 			surf,
 			shader=interp,
 			]
 			{
 				0.5*max(
 				1,
 				sqrt(
 				(2*0.5*(2*x-1)-1)^2
 				+ (4*0.5*y)^2
 				)
 				)
 			};
 		\end{axis}
 	\end{tikzpicture}
 	\caption{
 		The sharp bound
 		\[
 		B(\alpha,\mu)
 		=(1-\alpha)\max\left\{
 		1,\left|2(1-\alpha)(2\mu-1)-1\right|
 		\right\}
 		\]
 		as a function of the complex parameter
 		$\mu=\Re\mu+i\Im\mu$, shown for $\alpha=\frac12$.
 		The plot displays the full dependence on both the real and imaginary
 		parts of $\mu$.
 	}
 	\label{fig:complex-mu-bound}
 \end{figure}
 \begin{figure}[H]
 	\centering
 	\begin{tikzpicture}[scale=2.4]
 		\draw[->] (-0.2,0) -- (2.2,0)
 		node[right] {$\Re\mu$};
 		\draw[->] (0,-1.1) -- (0,1.1)
 		node[above] {$\Im\mu$};
 		
 		\fill[blue!12] (1,0) circle (0.5);
 		\draw[blue,very thick] (1,0) circle (0.5);
 		
 		\fill (1,0) circle (0.025);
 		\node[below] at (1,0) {$1$};
 		
 		\node[blue!70!black,align=center] at (1,0.2)
 		{$B(\tfrac12,\mu)=\tfrac12$};
 		
 		\node[align=center] at (1.72,0.72)
 		{$B(\tfrac12,\mu)=|\mu-1|$};
 		
 		\draw[dashed] (1,0) -- (1.5,0);
 		\node[below] at (1.25,0) {$\tfrac12$};
 		
 		\foreach \x in {0.5,1,1.5,2}
 		\draw (\x,0.025)--(\x,-0.025);
 	\end{tikzpicture}
 	\caption{
 		Geometry of the sharp bound for $\alpha=\frac12$.
 		The bound is constant and equal to $\frac12$ in the disk
 		$|\mu-1|\le\frac12$, and equals $|\mu-1|$ outside that disk.
 	}
 	\label{fig:complex-mu-levelset}
 \end{figure}
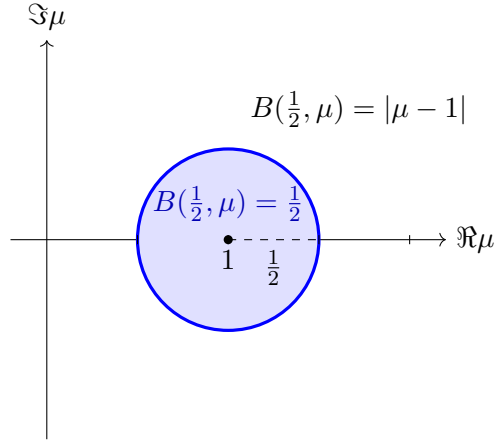
 
 \begin{theo}\label{Th-2}Let $f\in\mathcal{H}(\mathbb{U}^n,\mathbb{C})$ with $f(0)=1$ and $F(z)=zf(z)\in\mathcal{S}^*_{\alpha}(\mathbb{U}^n)$ and let $\mu$ be any complex number. Then
 	\begin{align*}
 	\displaystyle
 		\left\|\frac{D^3F(0)(z^3)}{3!}-\mu\frac{D^2F(0)}{2!}\left(z,\frac{D^2F(0)(z^2)}{2!}\right)\right\|_{\infty}\leq 
 		(1-\alpha)\max\{1,|2(1-\alpha)(2\mu-1)-1|\}||z||_{\infty},
 	\end{align*}
 	for all $z\in\mathbb{U}^n$. The inequality is sharp.
 \end{theo}
 
 \begin{rem} In the special case $X = \mathbb{C}$ and $B = \mathbb{U}$, Theorems \ref{Th-1} and \ref{Th-2} reduce to Theorem A.
 \end{rem}
 
 \subsection{{Auxiliary lemmas}}
 Let $\mathcal{P}$ denote the class of analytic functions $p$ in the open unit disk $\mathbb{U}$ satisfying $p(0) = 1$ and $\Re\{p(z)\}> 0$ for all $z \in \mathbb{U}$. A member of $\mathcal P$ is called a Carath\'eodory function. Coefficient bounds for functions in $\mathcal{P}$ play a fundamental role in solving various coefficient problems. The lemmas below are instrumental in establishing our main results.
\begin{lem}\label{Lem1} \emph{\cite{Cho-Kumar-Ravichandran-2019}}
Let $\displaystyle p(z)=1+\sum_{n=1}^{\infty}p_nz^n\in\mathcal{P}$. Then 
\begin{align*}
\displaystyle
|p_n|\leq 2\quad \text{and}\quad |p_n-p_mp_{n-m}|\leq 2,\quad n\geq 2,\;1\leq m\leq n-1.
\end{align*}
\end{lem}
\begin{lem}\label{Lem2}\emph{\cite[Problem 2.1.9]{Graham-Kohr}}
Let $\displaystyle p(z)=1+\sum_{n=1}^{\infty}p_nz^n\in\mathcal{P}$. Then
\begin{align*}
\displaystyle
\left|p_2-\frac{1}{2}p_1^2\right|\leq 2-\frac{1}{2}|p_1|^2.
\end{align*}	
\end{lem}

\begin{proof}[\bf Proof of Theorem \ref{Th-1}]
By the given condition, we have $F\in\mathcal{S}^*_{\alpha}(\mathbb{B})$, where $0\leq \alpha<1$. Using \eqref{SL1}, the condition of starlikeness of order $\alpha$ can be expressed as
\begin{align*}
\displaystyle
\Re\left\lbrace \frac{\|z\|}{l_z\left([D F(z)]^{-1}F(z)\right)}\right\rbrace>\alpha, \quad z\in\mathbb{B}\backslash \{0\},\quad l_z\in T(z).
\end{align*}

Let $z\in \mathbb{B}\backslash\{0\}$ be fixed and let $z_0=z/\|z\|$. We define the function $g:\mathbb{U}\to \mathbb{C}$ as follows
\begin{align*}	
\displaystyle
	g(\zeta)=\begin{cases}
		\frac{\zeta}{l_{z_0}\left([D F(\zeta z_0)]^{-1}F(\zeta z_0)\right)},& \zeta\neq 0,\\[8pt]
		 1& \zeta=0.
	\end{cases}
\end{align*}

Clearly $g\in\mathcal{H}(\mathbb{U})$ and $g(0)=1$.
We know that there is a one-to-one correspondence between $T(\zeta z)$ and $T(z)$ given by $l_{\zeta z}(.)=\frac{|\zeta|}{\zeta}l_z$, for each $\zeta\in\mathbb{C}\backslash\{0\}$. Therefore, we have
\begin{align*}
\displaystyle
g(\zeta)=\frac{\zeta}{l_{z_0}\left([D F(\zeta z_0)]^{-1}F(\zeta z_0)\right)}=\frac{\|\zeta z_0\|}{l_{\zeta z_0}\left([D F(\zeta z_0)]^{-1}F(\zeta z_0)\right)},\qquad \text{for} \ \zeta\in\mathbb{U}\backslash\{0\}.
\end{align*}	
Consequently $\Re\{g(\zeta)\}>\alpha$, for all $\zeta\in\mathbb{U}$.
Let \begin{align}\label{Th-1.1}
	h(\zeta)=\frac{g(\zeta)-\alpha}{1-\alpha}, \qquad\text{for all} \ \zeta\in\mathbb{U}.
	\end{align}
	Clearly, from \eqref{Th-1.1} we conclude that $h\in\mathcal{H}(\mathbb{U})$ with $\Re\{h(\zeta)\}>0$, for all $\zeta\in\mathbb{U}$ and $h(0)=1$.
Using a similar method as in Theorem 7.1.14 \cite{Graham-Kohr}, we have
\begin{align}\label{Th-1.2b}
	[D F(z)]^{-1}\eta=\frac{1}{f(z)}\left\lbrack \eta-\frac{\frac{zDf(z)}{f(z)}}{1+\frac{Df(z)z \eta}{f(z)}}\right\rbrack,\quad \eta\in X.
\end{align}

Therefore 
\begin{align}\label{Th-1.2a}
	[D F(z)]^{-1}F(z)=\frac{zf(z)}{f(z)+D f(z)z},
\end{align}
which implies that
\begin{align*}
\displaystyle
	\frac{\|z\|}{l_z\left([D F(z)]^{-1}F(z)\right)}=1+\frac{Df(z) z}{f(z)}
\end{align*}
and so
\begin{align}\label{Th-1.2}
\displaystyle
g(\zeta)=\frac{\|\zeta z_0\|}{l_{\zeta z_0}\left([D F(\zeta z_0)]^{-1}F(\zeta z_0)\right)}=1+\frac{Df(\zeta z_0) \zeta z_0}{f(\zeta z_0)}.
\end{align}
From the series expansions of $h(\zeta)$ and $f(\zeta z_0)$, it follows that
\begin{align}\label{Th-1.3}
\displaystyle
h(\zeta)=1+h'(0)\zeta+\frac{h''(0)}{2}\zeta^2+\ldots
\end{align}
and
\begin{align}\label{Th-1.4}
\displaystyle
f(\zeta z_0)=1+Df(0)(z_0)\zeta+\frac{1}{2!}D^2 f(0)(z_0^2)\zeta^2+\ldots.
\end{align}

Clearly
\begin{align}\label{Th-1.5}
\displaystyle
Df(\zeta z_0)(\zeta z_0)=Df(0)(z_0)\zeta+D^2 f(0)(z_0^2)\zeta^2+\ldots
\end{align}

Now using \eqref{Th-1.2}-\eqref{Th-1.5} to \eqref{Th-1.1}, we get
\begin{align}\label{Th-1.6}
\displaystyle
&\left(1+(1-\alpha)h'(0)\zeta+(1-\alpha)\frac{h''(0)}{2}\zeta^2+\ldots\right)\times\\&\left(1+Df(0)(z_0)\zeta+\frac{1}{2!}D^2 f(0)(z_0^2)\zeta^2+\ldots\right)\nonumber\\=&
\left(1+Df(0)(z_0)\zeta+\frac{1}{2!}D^2 f(0)(z_0^2)\zeta^2+\ldots\right)+\left(Df(0)(z_0)\zeta+D^2 f(0)(z_0^2)\zeta^2+\ldots\right).\nonumber
\end{align}

Clearly from \eqref{Th-1.6}, we find that
\begin{align}\label{Th-1.7}
\displaystyle
(1-\alpha)h'(0)=Df(0)(z_0)\;\;\text{and}\;\;(1-\alpha)\frac{h''(0)}{2}=D^2f(0)(z_0^2)-\left(Df(0)(z_0)\right)^2
\end{align}

Since $z_0=z/\|z\|$, from \eqref{Th-1.7}, we obtain 
\begin{align}\label{Th-1.9a}
	\displaystyle
	\frac{D^2f(0)(z^2)}{2!}&=\|z\|^2\frac{1}{2}\left((1-\alpha)\frac{h''(0)}{2}+((1-\alpha)h'(0))^2\right),\\\label{Th-1.10a}
	Df(0)(z))&=||z||(1-\alpha)h'(0).
\end{align}

Again since $F(z)=zf(z)$, it follows that
\begin{align*}
\displaystyle
\frac{D^3F(0)(z^3)}{3!}=\frac{D^2f(0)(z^2)}{2!}z\quad \text{and}\quad \frac{D^2F(0)(z^2)}{2!}=Df(0)(z) z.
\end{align*}

Consequently
\begin{align}\label{Th-1.8}
\displaystyle
	\frac{l_{z}(D^3F(0)(z^3))}{3!}=\frac{D^2f(0)(z^2)}{2!}\|z\|\quad \text{and}\quad \frac{l_z(D^2F(0)(z^2))}{2!}=Df(0)(z) \|z\|.
\end{align}

Using \eqref{Th-1.8} to \eqref{Th-1.9a} and \eqref{Th-1.10a}, we get 
\begin{align}\label{Th-1.9}
\displaystyle
	\frac{l_{z}(D^3F(0)(z^3))}{3!\|z\|^3}&=\frac{1}{2}\left((1-\alpha)\frac{h''(0)}{2}+((1-\alpha)h'(0))^2\right),\\\label{Th-1.10}
 \frac{l_z(D^2F(0)(z^2))}{2!\|z\|^2}&=(1-\alpha)h'(0).
\end{align}

Now from \eqref{Th-1.9} and \eqref{Th-1.10}, we deduce that 
\begin{align*}
\displaystyle
&\left|\frac{l_{z}(D^3F(0)(z^3))}{3!\|z\|^3}-\mu\left(\frac{l_z(D^2F(0)(z^2))}{2!\|z\|^2}\right)^2\right|\\=&
\frac{1}{2}\left|(1-\alpha)\frac{h''(0)}{2}-(2\mu-1)((1-\alpha)h'(0))^2\right|\\=&
\frac{1-\alpha}{2}\left|\frac{h''(0)}{2}-(2\mu-1)(1-\alpha)( h'(0))^2\right|\\=&
\frac{1-\alpha}{2}\left|\frac{h''(0)}{2}-\frac{1}{2}(h'(0))^2+\frac{1-2(1-\alpha)(2\mu-1)}{2}(h'(0))^2\right|\\\leq&
\frac{1-\alpha}{2}\left(\left|\frac{h''(0)}{2}-\frac{1}{2}(h'(0))^2\right|+\frac{|2(1-\alpha)(2\mu-1)-1|}{2}\left|(h'(0))^2\right|\right).
\end{align*}
Therefore, using Lemma \ref{Lem2}, we obtain
\begin{align}\label{Th-1.11}
\displaystyle
&\left|\frac{l_{z}(D^3F(0)(z^3))}{3!\|z\|^3}-\mu\left(\frac{l_z(D^2F(0)(z^2))}{2!\|z\|^2}\right)^2\right|\\\leq& 
	\frac{1-\alpha}{2}\left(2-\frac{1}{2}|h'(0)|^2+\frac{|2(1-\alpha)(2\mu-1)-1|}{2}|h'(0)|^2\right).\nonumber
	\end{align}
On the other hand, by Lemma \ref{Lem1}, we know that $|h'(0)|\in [0,2]$. Therefore maximizing the expression \eqref{Th-1.11} in $|h'(0)|$ over $[0,2]$ gives the result.\vspace{1.2mm}

To establish the sharpness of the bound, let us consider the function $F$ defined by
\begin{align*}
\displaystyle
	F(z)=zf(z)=\frac{z}{\left(1-l_u(z)\right)^{2(1-\alpha)}},\quad z\in\mathbb{B},\quad \|u\|=1.
\end{align*}

According to Lemma 2.2 \cite{Xu-He-Xu}, it is clear that $F\in\mathcal{S}^*_{\alpha}(\mathbb{B})$. Note that
\begin{align*}
f(z)=&1+2(1-\alpha)l_u(z)+\frac{2(1-\alpha)(3-2\alpha)}{2!}(l_u(z))^2\\&+\frac{2(1-\alpha)(3-2\alpha)(4-2\alpha)}{3!}(l_u(z))^3+\ldots.
\end{align*}
Since, $F(z)=zf(z)$, we have
\begin{align*}
	F(z)=&z+2(1-\alpha)l_u(z)z+\frac{2(1-\alpha)(3-2\alpha)}{2!}(l_u(z))^2z\\&+\frac{2(1-\alpha)(3-2\alpha)(4-2\alpha)}{3!}(l_u(z))^3z+\ldots
\end{align*}
Consequently
\begin{align*}
\displaystyle
\frac{D^2F(0)(z^2)}{2!}=2(1-\alpha)l_u(z)z\quad \text{and}\quad \frac{D^3F(0)(z^3)}{3!}=(1-\alpha)(3-2\alpha)(l_u(z))^2z.
\end{align*}
Therefore
\begin{align*}
\displaystyle
	\frac{l_z(D^2F(0)(z^2))}{2!}=2(1-\alpha)l_u(z)\|z\|
\;\text{and}\;
 \displaystyle
 \frac{l_z(D^3F(0)(z^3))}{3!}=(1-\alpha)(3-2\alpha)(l_u(z))^2\|z\|.
\end{align*}
Taking $z=ru\;(0<r<1)$, we get
\begin{align*}
\displaystyle
	\frac{l_z(D^2F(0)(z^2))}{2!\|z\|^2}=2(1-\alpha)\quad \text{and}\quad
	\frac{l_z(D^3F(0)(z^3))}{3!\|z\|^3}=(1-\alpha)(3-2\alpha).
\end{align*}
Hence
\begin{align*}
\displaystyle
	\left|\frac{l_{z}(D^3F(0)(z^3))}{3!\|z\|^3}-\mu\left(\frac{l_z(D^2F(0)(z^2))}{2!\|z\|^2}\right)^2\right|=(1-\alpha)\max\{1,|2(1-\alpha)(2\mu-1)-1|\}.
\end{align*}
This proves the desired sharp estimate.
\end{proof}
\begin{figure}[H]
	\centering
	\begin{tikzpicture}
		\begin{axis}[
			width=12.5cm,
			height=7.5cm,
			xmin=0,
			xmax=0.95,
			ymin=0,
			ymax=11,
			xlabel={$\alpha$},
			ylabel={$\mu$},
			axis lines=left,
			grid=both,
			major grid style={gray!30,line width=0.35pt},
			minor grid style={gray!15,line width=0.2pt},
			minor x tick num=1,
			minor y tick num=1,
			xtick={0,0.2,0.4,0.6,0.8},
			ytick={0,1,2,3,4,5,6,7,8,9,10},
			legend style={
				at={(0.98,0.98)},
				anchor=north east,
				font=\small,
				draw=black!40,
				fill=white
			},
			clip=false
			]
			
			\addplot[
			name path=lower,
			domain=0:0.95,
			very thick
			]
			{0.5};
			\addlegendentry{$\mu=\frac12$}
			
			\addplot[
			name path=upper,
			domain=0:0.95,
			very thick,
			dashed
			]
			{0.5+1/(2*(1-x))};
			\addlegendentry{
				$\displaystyle
				\mu=\frac12+\frac{1}{2(1-\alpha)}$
			}
			
			\addplot[
			fill=blue!12,
			draw=none,
			domain=0:0.95
			]
			{0.5+1/(2*(1-x))}
			\closedcycle;
			
			\node[
			align=center,
			font=\small
			] at (axis cs:0.42,2.0)
			{$A(\alpha,\mu)\leq1$\\
				\scriptsize $B(\alpha,\mu)=1-\alpha$};
			
			\node[
			align=center,
			font=\small
			] at (axis cs:0.35,7.2)
			{$A(\alpha,\mu)>1$\\
				\scriptsize
				$B(\alpha,\mu)
				=(1-\alpha)A(\alpha,\mu)$};
			
			\fill (axis cs:0,0.5) circle (2pt);
			\fill (axis cs:0,1.0) circle (2pt);
			
		\end{axis}
	\end{tikzpicture}
	
	\caption{
		Parameter regions associated with the two branches of the sharp
		bound
		\[
		B(\alpha,\mu)
		=(1-\alpha)
		\max\left\{
		1,\,
		\left|2(1-\alpha)(2\mu-1)-1\right|
		\right\},
		\qquad 0\leq\alpha<1.
		\]
		For real $\mu$, the transition curves are
		$\mu=\frac12$ and
		$\mu=\frac12+\frac{1}{2(1-\alpha)}$.
		Between these curves the bound equals $1-\alpha$,
		while outside this region the second branch is active.
	}
	\label{fig:parameter-regions}
\end{figure}
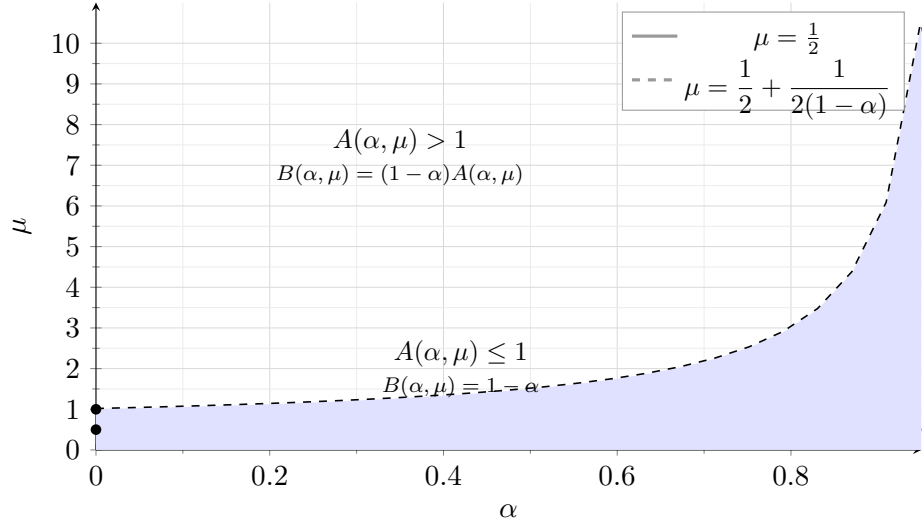
\begin{proof}[\bf Proof of Theorem \ref{Th-2}]
In view of the assumption $F\in\mathcal{S}^*_{\alpha}(\mathbb{U}^n)$, where $0\leq \alpha<1$, we have
\begin{align}\label{Th-2.1}
\displaystyle
		\Re\left\lbrace \frac{\|z\|_{\infty}}{l_z\left([D F(z)]^{-1}F(z)\right)}\right\rbrace>\alpha, \quad z\in\mathbb{U}^n\backslash \{0\},\quad l_z\in T(z).
	\end{align}
Let $w(z)=(w_1(z),w_2(z),\ldots,w_n(z))'=[D F(z)]^{-1}F(z)$ be the column vector in $\mathbb{C}^n$. According to (\ref{eq3}), we have
\begin{align}\label{Th-2.2}
\displaystyle
l_z(w)=\sum\limits_{\{j:||z||_{\infty}=|z_j|\}}t_j\frac{w_j(z)\ol z_j}{||z||_{\infty}},
\end{align}
where each $t_j\geq 0$ and $\sum\limits_{\{j:||z||_{\infty}=|z_j|\}}t_j=1.$\\
Let $k\in\{1,2,\ldots,n\}$ be such that $||z||_{\infty}=|z_k|$. Choose the admissible functional
corresponding to $t_k=1$ and $t_j=0$, where $j\in\{1,2,\ldots,n\}$ such that $j\neq k$. Then from \eqref{Th-2.1} and \eqref{Th-2.2}, we have
\begin{align}\label{Th-2.3}
\displaystyle
\Re\left\lbrace \frac{z_k}{w_k(z)}\right\rbrace>\alpha,\quad z\in\mathbb{U}^n\backslash\{0\},
\end{align}
where $w_k(z)$ denotes the $k$th coordinate of $w(z)$ corresponding to an index $k$ satisfying $\|z\|_\infty=|z_k|$.
Let $z\in\mathbb{U}^n\backslash\{0\}$ and let $z_0=z/\|z\|_{\infty}$. We define the function $g_k:\mathbb{U}\to \mathbb{C}$ as follows
\begin{align}\label{Th-2.4}	
\displaystyle
	g_k(\zeta)=\begin{cases}
		\frac{\zeta z_k}{w_k(\zeta z_0)},& \zeta\neq 0,\\[8pt]
		1& \zeta=0.
	\end{cases}
\end{align}

Consequently, from \eqref{Th-2.3} and \eqref{Th-2.4}, we have $\Re\{g_k(\zeta)\}>0$ 
for all $\zeta\in\mathbb{U}$. 
Let
\begin{align}\label{Th-2.5}
\displaystyle
	h_k(\zeta)=\frac{g_k(\zeta)-\alpha}{1-\alpha},\qquad\text{for all} \ \zeta\in\mathbb{U}.
\end{align}
Since $\Re\{g_k(\zeta)\}>0$ 
for all $\zeta\in\mathbb{U}$, it follows that from \eqref{Th-2.5}, we obtain $h_k\in\mathcal{H}(\mathbb{U})$ with $\Re\{h_k(\zeta)\}>0$ for all $\zeta\in\mathbb{U}$ and $h_k(0)=1$.
On the other hand, from \eqref{Th-1.2a}, we have
\begin{align*}
\displaystyle
w_k(z)=z_k\left\lbrack \frac{f(z)}{f(z)+Df(z)z}\right\rbrack,
\end{align*}
which implies that
\begin{align}\label{Th-2.6}
\displaystyle
	w_k(\zeta z_0)=\zeta z_k\left\lbrack \frac{f(\zeta z_0)}{f(\zeta z_0)+Df(\zeta z_0)\zeta z_0}\right\rbrack.
\end{align}

Substituting \eqref{Th-2.4} and \eqref{Th-2.6} into \eqref{Th-2.5}, we get
\begin{align}\label{Th-2.7}
\displaystyle
\left((1-\alpha)h_k(\zeta)+\alpha\right)f(\zeta z_0)=f(\zeta z_0)+Df(\zeta z_0)\zeta z_0.	
\end{align} 

Expanding both sides of \eqref{Th-2.7} into homogeneous polynomials and equating coefficients of
homogeneous terms we obtain
\begin{align}\label{Th-2.8}
\displaystyle
	(1-\alpha)h_k'(0)=Df(0)(z_0)\;\;\text{and}\;\;(1-\alpha)\frac{h_k''(0)}{2}=D^2f(0)(z_0^2)-\left(Df(0)(z_0)\right)^2.
\end{align}

Now from \eqref{Th-2.8}, we obtain 
\begin{align}\label{Th-2.9}
\displaystyle
	\frac{D^2f(0)(z_0^2)}{2!}=\frac{1}{2}\left((1-\alpha)\frac{h_{k}^{''}(0)}{2}+((1-\alpha)h_{k}^{'}(0))^2\right)\ \text{and} \ Df(0)(z_0)=(1-\alpha)h_{k}^{'}(0).
\end{align}
Let $F(z)=(F_1(z),F_2(z),\ldots,F_n(z))$. Since $F(z)=zf(z)$, it follows that
\begin{align*}
\displaystyle
	\frac{D^3F_k(0)(z^3)}{3!}=\frac{D^2f(0)(z^2)}{2!}z_k\quad \text{and}\quad \frac{D^2F_k(0)(z^2)}{2!}=Df(0)(z) z_k.
\end{align*}

Since $z_0=z/\|z\|_\infty$, it follows that
\begin{align}\label{Th-2.11}
\displaystyle
	\frac{D^3F_k(0)(z_0^3)}{3!}=\frac{D^2f(0)(z_0^2)}{2!}\frac{z_k}{||z||_{\infty}}\quad \text{and}\quad \frac{D^2F_k(0)(z_0^2)}{2!}=Df(0)(z_0) \frac{z_k}{||z||_{\infty}}.
\end{align}

In view of \eqref{Th-2.11}, we obtain
\begin{align}\label{Th-2.12}
\displaystyle
&\left|\mu\frac{1}{2!}D^2F_k(0)\left(z_0,\frac{D^2F(0)(z_0^2)}{2!}\right)\frac{||z||_{\infty}}{z_k}-\frac{D^3F_k(0)(z_0^3)}{3!}\frac{||z||_{\infty}}{z_k}\right|\\=&
\left|\mu\frac{1}{2!}D^2F_k(0)\left(z_0,Df(0)(z_0)z_0\right)\frac{||z||_{\infty}}{z_k}-\frac{D^2f(0)(z_0^2)}{2!}\right|\nonumber\\=&
\left|\mu Df(0)(z_0)\frac{1}{2!}D^2F_k(0)\left(z_0,z_0\right)\frac{||z||_{\infty}}{z_k}-\frac{D^2f(0)(z_0^2)}{2!}\right|
\nonumber\\=&
\left|\mu Df(0)(z_0)\frac{1}{2!}D^2F_k(0)\left(z_0^2\right)\frac{||z||_{\infty}}{z_k}-\frac{D^2f(0)(z_0^2)}{2!}\right|
\nonumber\\=&
\left|\mu\left(Df(0)(z_0)\right)^2-\frac{D^2f(0)(z_0^2)}{2!}\right|.\nonumber
\end{align}

Now using \eqref{Th-2.9} to \eqref{Th-2.12}, we get
\begin{align}\label{k2}
\displaystyle
&\left|\frac{D^3F_k(0)(z_0^3)}{3!}\frac{||z||_{\infty}}{z_k}-\mu\frac{1}{2!}D^2F_k(0)\left(z_0,\frac{D^2F(0)(z_0^2)}{2!}\right)\frac{||z||_{\infty}}{z_k}\right|\\=&
\frac{1}{2}\left|(1-\alpha)\frac{h''(0)}{2}-(2\mu-1)((1-\alpha)h'(0))^2\right|\nonumber\\\leq&
\frac{1-\alpha}{2}\left(\left|\frac{h''(0)}{2}-\frac{1}{2}(h'(0))^2\right|+\frac{|2(1-\alpha)(2\mu-1)-1|}{2}\left|(h'(0))^2\right|\right).\nonumber
\end{align}
Applying Lemma \ref{Lem2} in \eqref{k2}, we obtain
\begin{align}\label{Th-2.13}
\displaystyle
&\left|\frac{D^3F_k(0)(z_0^3)}{3!}\frac{||z||_{\infty}}{z_k}-\mu\frac{1}{2!}D^2F_k(0)\left(z_0,\frac{D^2F(0)(z_0^2)}{2!}\right)\frac{||z||_{\infty}}{z_k}\right|\\\leq& 
\frac{1-\alpha}{2}\left(2-\frac{1}{2}|h'(0)|^2+\frac{|2(1-\alpha)(2\mu-1)-1|}{2}|h'(0)|^2\right).\nonumber
\end{align}

On the other hand, by Lemma \ref{Lem1}, we know that $|h'(0)|\in [0,2]$. Hence maximization of the expression \eqref{Th-2.13} in $|h'(0)|$ over $[0,2]$ yields (see Figure \ref{fig:optimization}).
\begin{figure}[H]
	\centering
	
	\begin{tikzpicture}
		\begin{axis}[
			width=10cm,
			height=7.6cm,
			xmin=0,
			xmax=2.05,
			ymin=1.0,
			ymax=4.3,
			axis lines=left,
			xlabel={$x=|h'(0)|$},
			ylabel={$\Phi(x)$},
			xtick={0,0.5,1,1.5,2},
			ytick={1,1.5,2,2.5,3,3.5,4},
			minor x tick num=1,
			minor y tick num=1,
			grid=both,
			major grid style={gray!35,line width=0.3pt},
			minor grid style={gray!15,line width=0.2pt},
			legend style={
				draw=black!40,
				fill=white,
				rounded corners=1pt,
				font=\scriptsize,
				legend pos=outer north east,
				row sep=2pt,
			},
			samples=150,
			domain=0:2,
			every axis plot/.append style={very thick},
			clip=false,
			]
			
			\foreach \A in {0.7,0.85,1.15,1.3,1.5,1.7}{
				\addplot[gray!55,thin,samples=60,domain=0:2,forget plot]
				{2+((\A-1)/2)*x^2};
			}
			\addlegendimage{gray!55,thin}
			\addlegendentry{$A\in(0,2)$ (family)}
			
				\addplot[blue,fill=blue,fill opacity=0.08,forget plot]
			{2-0.25*x^2} \closedcycle;
			
			\addplot[blue]
			{2-0.25*x^2};
			\addlegendentry{$A=0.5$}
			
			\addplot[black,dashed]
			{2};
			\addlegendentry{$A=1$}
			
			\addplot[red,fill=red,fill opacity=0.08,forget plot]
			{2+0.5*x^2} \closedcycle;
			
			\addplot[red]
			{2+0.5*x^2};
			\addlegendentry{$A=2$}

			\fill[blue]
			(axis cs:0,2)
			circle(2.6pt);
			
			\fill[red]
			(axis cs:2,4)
			circle(2.6pt);
			
			\node[blue,above right,inner sep=2pt]
			at (axis cs:0,2.2)
			{\scriptsize Maximum};
			
			\node[red,above left,inner sep=2pt]
			at (axis cs:2,4)
			{\scriptsize Maximum};
			
			\node
			at (axis cs:0.62,2.14)
			{\scriptsize constant};
			
			\draw[->,blue]
			(axis cs:0.4,2.3)
			--
			(axis cs:0.02,2.07);
			
			\draw[->,red]
			(axis cs:1.62,3.55)
			--
			(axis cs:1.95,3.96);
			
		\end{axis}
	\end{tikzpicture}
	
	\caption{
	The function $\Phi_A(x)=2+\frac{A-1}{2}x^2,
	\qquad 0\le x\le2,$ where
	$	A=\left|2(1-\alpha)(2\mu-1)-1\right|.$
	If $A<1$, the maximum is attained at $x=0$; if $A=1$,
	the function is constant; and if $A>1$, the maximum is attained at $x=2$. Thus $\max_{0\le x\le2}\Phi_A(x)=2\max\{1,A\}.$	}
	
	\label{fig:optimization}
	
\end{figure}
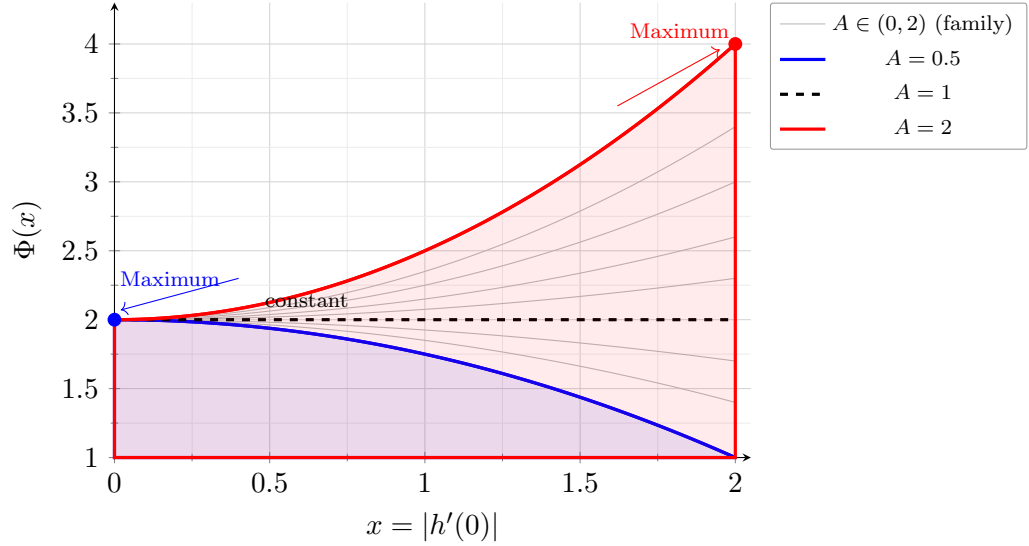

\begin{align*}
\displaystyle
	&\left|\frac{D^3F_k(0)(z_0^3)}{3!}\frac{||z||_{\infty}}{z_k}-\mu\frac{1}{2!}D^2F_k(0)\left(z_0,\frac{D^2F(0)(z_0^2)}{2!}\right)\frac{||z||_{\infty}}{z_k}\right|\\\leq& 
	(1-\alpha)\max\{1,|2(1-\alpha)(2\mu-1)-1|\}.\nonumber
\end{align*}
Therefore, it follows that
\begin{align}\label{Th-2.13a}
\displaystyle
&\left|\frac{D^3F_k(0)(z_0^3)}{3!}-\mu\frac{1}{2!}D^2F_k(0)\left(z_0,\frac{D^2F_k(0)(z_0^2)}{2!}\right)\right|\\\leq& 
(1-\alpha)\max\{1,|2(1-\alpha)(2\mu-1)-1|\}.\nonumber
\end{align}

In particular, when $z_0\in\partial_0 \mathbb{U}^n$, then from \eqref{Th-2.13a}, we deduce that
\begin{align}\label{Th-2.13b}
	\displaystyle
	&\left|\frac{D^3F_j(0)(z_0^3)}{3!}-\mu\frac{1}{2!}D^2F_j(0)\left(z_0,\frac{D^2F(0)(z_0^2)}{2!}\right)\right|\\\leq& 
	(1-\alpha)\max\{1,|2(1-\alpha)(2\mu-1)-1|\},\quad j=1,2,\ldots,n.\nonumber
\end{align}

We know that
\begin{align*}
	\frac{D^3F_j(0)(z^3)}{3!}-\mu\frac{1}{2!}D^2F_j(0)\left(z,\frac{D^2F(0)(z^2)}{2!}\right)\in\mathcal{H}(\ol{\mathbb{U}^n}),\quad j=1,2,\ldots,n.
	\end{align*}

By the maximum modulus theorem for holomorphic functions on the unit polydisk, we deduce from \eqref{Th-2.13b} that
\begin{align*}
\displaystyle
	&\left|\frac{D^3F_j(0)(z_0^3)}{3!}-\mu\frac{1}{2!}D^2F_j(0)\left(z_0,\frac{D^2F(0)(z_0^2)}{2!}\right)\right|\\\leq& 
	(1-\alpha)\max\{1,|2(1-\alpha)(2\mu-1)-1|\}||z||_{\infty}\nonumber
\end{align*}
for all $z_0\in\partial \mathbb{U}^n$, where $j=1,2,\ldots,n$.
Hence
\begin{align*}
	\displaystyle
	&\left|\frac{D^3F(0)(z_0^3)}{3!}-\mu\frac{1}{2!}D^2F(0)\left(z_0,\frac{D^2F(0)(z_0^2)}{2!}\right)\right|\\\leq& 
	(1-\alpha)\max\{1,|2(1-\alpha)(2\mu-1)-1|\}||z||_{\infty}\nonumber
\end{align*}
for all $z_0\in\partial \mathbb{U}^n$ and so
\begin{align}\label{Th-2.14}
\displaystyle
	&\left\|\frac{D^3F(0)(z_0^3)}{3!}-\mu\frac{1}{2!}D^2F(0)\left(z_0,\frac{D^2F(0)(z_0^2)}{2!}\right)\right\|_{\infty}\\\leq& 
	(1-\alpha)\max\{1,|2(1-\alpha)(2\mu-1)-1|\}||z||_{\infty}\nonumber
\end{align}
for all $z\in\mathbb{U}^n$.

\vspace{1.2mm}

To verify the sharpness of the bound, consider the function
\begin{align*}
\displaystyle
F(z)=zf(z)=\left(\frac{z_1}{(1-z_1)^{2(1-\alpha)}},\frac{z_2}{(1-z_1)^{2(1-\alpha)}},\ldots,\frac{z_n}{(1-z_1)^{2(1-\alpha)}}\right).
\end{align*}

It is clear that $F\in\mathcal{S}^*_{\alpha}(\mathbb{U}^n)$. Also we have
\begin{align*}
\displaystyle
	\frac{D^2F(0)(z^2)}{2!}=2(1-\alpha)z_1z\quad \text{and}\quad \frac{D^3F(0)(z^3)}{3!}=(1-\alpha)(3-2\alpha)(1-z_1)^2z.
\end{align*}

Taking $z=(r,0,\ldots,0)$, where $0<r<1$ a direct computation shows that the equality of \eqref{Th-2.14} holds, which establishes the sharpness of the bound.
\end{proof}

Next, by removing the restrictive assumption $F(z) = f(z)z$, we generalize Theorem A to higher dimensions under weaker conditions than those in Theorem \ref{Th-1}.

\begin{theo}\label{Th-1a}
Let $F$ be a locally biholomorphic mapping on $\mathbb{B}$, and $F$ satisfy the following assumptions:
\begin{align}\label{Th1a}
\displaystyle
\frac{D^{k+1}F(0)(z^{k+1})}{(k + 1)!} = H_{F,k}(z)z, \quad z \in \mathbb{B}, \; k = 1, 2, 3,
\end{align}
where $H_{F,k}(z)$ is a homogeneous polynomial of degree $k$ with values in $\mathbb{C}$. Let $\mu$ be any complex number. If $F\in\mathcal{S}^*_{\alpha}(\mathbb{B})$, then the inequality \eqref{Th-A} holds and the bound is sharp.
\end{theo}
\begin{proof} Let $z\in \mathbb{B}\backslash\{0\}$ be fixed and let $z_0=z/\|z\|_\infty$. Now proceeding in the same way as done in the proof of Theorem \ref{Th-1}, we can prove that $h\in\mathcal{H}(\mathbb{U})$ with $\Re\{h(\zeta)\}>0$ for all $\zeta\in\mathbb{U}$ and $h(0)=1$, where $h(\zeta)$ is defined by \eqref{Th-1.1}. 
Let
	\begin{align}\label{Th-1a.2}
		\varphi(z)=z+\sum\limits_{n=2}^{\infty} \frac{1}{n!}D^n \varphi(0)(z^n),\quad z\in\mathbb{B},
	\end{align}
where $\varphi(z)=[DF(z)]^{-1}F(z)$ with $D \varphi(0)=I$.
Clearly from \eqref{Th-1a.2}, we have
\begin{align}\label{Th-1a.8}
	l_{z_0}(\varphi(\zeta z_0))=\zeta+\frac{l_{z_0}(D^2 \varphi(0)(z_0^2))}{2!}\zeta^2+\frac{l_{z_0}(D^3 \varphi(0)(z_0^3))}{3!}\zeta^3+\ldots.
\end{align}

As in the proof of Theorem~\ref{Th-1}, define
$g(\zeta)=\frac{\zeta}{l_{z_0}(\varphi(\zeta z_0))}$.\\
Using \eqref{Th-1.3} and \eqref{Th-1a.8} to \eqref{Th-1.1}, we deduce that
\begin{align*}
	&\left\lbrace1+(1-\alpha)h'(0)\zeta+(1-\alpha)\frac{h''(0)}{2}\zeta^2+\ldots \right\rbrace\times\\&\left\lbrace \zeta+\frac{l_{z_0}(D^2 \varphi(0)(z_0^2))}{2!}\zeta^2+\frac{l_{z_0}(D^3 \varphi(0)(z_0^3))}{3!}\zeta^3+\ldots\right\rbrace=\zeta.
\end{align*}

A straightforward computation yields
\begin{align}\label{Th-1a.9}
	\frac{l_{z_0}(D^2 \varphi(0)(z_0^2))}{2!}=&-(1-\alpha)h'(0),\\\label{Th-1a.10}
	\frac{l_{z_0}(D^3 \varphi(0)(z_0^3))}{3!}=&((1-\alpha)h'(0))^2-(1-\alpha)\frac{h''(0)}{2!}.
\end{align}
	
 Let
\begin{align}\label{Th-1a.1}
F(z)=z+\sum\limits_{n=2}^{\infty} \frac{1}{n!}D^n F(0)(z^n),\quad z\in\mathbb{B}.
\end{align}

Clearly
\begin{align}\label{Th-1a.3}
DF(z)=I+\sum\limits_{n=2}^{\infty} \frac{n}{n!}D^n F(0)(z^{n-1}),\quad z\in\mathbb{B}.
\end{align}

Now substituting \eqref{Th-1a.1}-\eqref{Th-1a.3} into $DF(z) \varphi(z)=F(z)$ and then comparing coefficients of equal powers of $z$ in the identity, we deduce that
\begin{align}\label{Th-1a.4}
\frac{D^2F(0)(z^2)}{2!}=&-\frac{D^2\varphi(0)(z^2)}{2!},\\\label{Th-1a.5}
-\frac{D^3F(0)(z^3)}{3}=&\frac{1}{3!}D^3\varphi(0)(z^3)+D^2F(0)\left(z,\frac{D^2\varphi(0)(z^2)}{2!}\right).
\end{align}

By assumptions \eqref{Th1a}, we deduce from \eqref{Th-1a.9} and \eqref{Th-1a.4} that
\begin{align}\label{Th-1a.7}
\frac{l_{z_0}(D^2F(0)(z_0^2))}{2!}=&H_{F,1}(z_0)=-\frac{l_{z_0}(D^2\varphi(0)(z_0^2))}{2!}=(1-\alpha)h'(0)\\\label{Th-1a.11}
	\frac{D^2F(0)(z_0^2)}{2!}=&H_{F,1}(z_0)z_0=-\frac{D^2\varphi(0)(z_0^2)}{2!}.
	\end{align}

Since $z_0=z/||z||$, \eqref{Th-1a.7} yields	
\begin{align}\label{Th-1a.12}
\frac{l_{z_0}(D^2F(0)(z^2))}{2!||z||^2}=(1-\alpha)h'(0).
\end{align}

Again by assumptions \eqref{Th1a}, we can deduce from \eqref{Th-1a.9},  \eqref{Th-1a.10}, \eqref{Th-1a.5} and \eqref{Th-1a.11} that
\begin{align*}
\frac{l_{z_0}(D^3F(0)(z_0^3))}{3!}=&H_{F,2}(z_0)\\=&
-\frac{1}{2}\left(\frac{l_{z_0}(D^3\varphi(0)(z_0^3))}{3!}+l_{z_0}\left(D^2F(0)\left(z_0,\frac{D^2\varphi(0)(z_0^2)}{2!}\right)\right)\right)\nonumber\\=&
-\frac{1}{2}\left(\frac{l_{z_0}(D^3\varphi(0)(z_0^3))}{3!}+l_{z_0}\left(D^2F(0)\left(z_0,-H_{F,1}(z_0)z_0\right)\right)\right)\nonumber\\=&
-\frac{1}{2}\left(\frac{l_{z_0}(D^3\varphi(0)(z_0^3))}{3!}-H_{F,1}(z_0)l_{z_0}\left(D^2F(0)(z_0^2)\right)\right)\\=&
-\frac{1}{2}\left(\frac{l_{z_0}(D^3\varphi(0)(z_0^3))}{3!}-2(H_{F,1}(z_0))^2\right).\end{align*} 
Hence we have, 
\begin{align}\label{Th-1a.13} &
\frac{l_{z_0}(D^3F(0)(z_0^3))}{3!}\nonumber\\=&
-\frac{1}{2}\left(\frac{l_{z_0}(D^3\varphi(0)(z_0^3))}{3!}-H_{F,1}(z_0)l_{z_0}\left(D^2F(0)(z_0^2)\right)\right)\nonumber\\=&
-\frac{1}{2}\left(\frac{l_{z_0}(D^3\varphi(0)(z_0^3))}{3!}-2\left(\frac{l_{z_0}(D^2\varphi(0)(z_0^2))}{2!}\right)^2\right)\nonumber\\=&
\frac{1}{2}\left((1-\alpha)h''(0)+((1-\alpha)h'(0))^2\right).
\end{align}
Combining \eqref{Th-1a.12} and \eqref{Th-1a.13} with the Carathéodory coefficient estimates exactly as in the proof of Theorem~\ref{Th-1}, we obtain inequality \eqref{Th-A}, which is the required conclusion. The sharpness follows immediately  from Theorem~\ref{Th-1}. This completes the proof.
\end{proof}

\section{{\bf Fekete-Szeg\"{o} inequalities for quasi-convex mappings of type $B$ on complex Banach space}}\label{Sec-3}
Let $X$ and $Y$ be two complex Banach spaces and $\Omega\subset X$ be a domain.
Also let $f:\Omega\rightarrow Y$ be a holomorphic mapping. We say that $f$ is
\emph{convex} on $\Omega$ if $f$ is biholomorphic on $\Omega$ and
$f(\Omega)$ is a convex domain in $Y$, that is, $(1-t)f(z)+tf(w)\in f(\Omega)$; for all $z, w\in\Omega,\; t\in[0,1]$.

A normalized locally biholomorphic mapping $f\in\mathcal{H}(\mathbb{B})$ is said to be quasi-convex of type $B$ if (see \cite[Definition 2.6]{Ham2})
\begin{align}\label{C1}
	\displaystyle
	\Re\left\lbrace l_z\left([Df(z)]^{-1}\left(D^2f(z)(z,z)+Df(z)z\right)\right)\right\rbrace>0,\quad z\in\mathbb{B}.
\end{align}

If $\mathbb{B}$ is the unit ball of $X$, we let $\mathcal{K}_B(\mathbb{B})$ denote the subset of $\mathcal{S}(\mathbb{B})$
consisting of normalized quasi-convex mappings of type $B$ from $\mathbb{B}$ into $X$.\vspace{1.2mm}

Numerous references now address convex mappings in higher dimensions. For a detailed overview, see the recent monographs by Gong \cite{Gong} and Kohr and Liczberski  \cite{Kohr- Liczberski-1998}. Key foundational papers in this area include \cite{Ham4}, \cite{Liu-Liu-2011}, \cite{Mat},
\cite{Roper-Suffridge-1999}-\cite{Su3}, \cite{Xu-Dl-Li}, \cite{Xu-Yang-Liu-Xu-2015}, \cite{Zhang-Liu-2002}.
\vspace{1.2mm}


In this section, we extend Theorem B to the case of complex Banach space.
\begin{theo}\label{Th-3}Let $f\in\mathcal{H}(\mathbb{B},\mathbb{C})$ with $f(0)=1$ and $F(z)=zf(z)\in\mathcal{K}_B(\mathbb{B})$ and let $\mu$ be any complex number. Then
	\begin{align}\label{Th3a}
	\displaystyle
		\left|\frac{l_{z}(D^3F(0)(z^3))}{3!\|z\|^3}-\mu\left(\frac{l_z(D^2F(0)(z^2))}{2!\|z\|^2}\right)^2\right|\leq 
		\frac{1}{3}\max\{1,3|1-\mu|\},
	\end{align}
for $z\in\mathbb{B}\backslash\{0\}$. The inequality \eqref{Th3a} is sharp.
\end{theo}
\begin{figure}[H]
	\centering
	\definecolor{flatfill}{RGB}{255,230,150}
	\definecolor{sidefillL}{RGB}{190,215,255}
	\definecolor{sidefillR}{RGB}{190,215,255}
	\begin{tikzpicture}
		\begin{axis}[
			width=11cm, height=7cm,
			xlabel={$\mu\in\mathbb{R}$},
			ylabel={$\varphi(\mu)$},
			xmin=-1, xmax=3.2, ymin=0, ymax=2.1,
			axis lines=middle,
			axis line style={-Latex,thick},
			xtick={-1,-0.5,0,0.5,2/3,1,4/3,1.5,2,2.5,3},
			xticklabels={$-1$,$-\tfrac12$,$0$,,,$1$,,,$2$,$\tfrac52$,$3$},
			ytick={0,1/3,0.5,1,1.5,2},
			yticklabels={$0$,,,$1$,,$2$},
			grid=both,
			major grid style={gray!22,line width=0.3pt},
			minor grid style={gray!10,line width=0.15pt},
			minor x tick num=1,
			minor y tick num=1,
			domain=-1:3, samples=200,
			thick,
			legend style={
				draw=black!45,
				fill=white,
				rounded corners=1pt,
				font=\scriptsize,
				legend pos=outer north east,
				row sep=3pt,
				legend cell align={left},
			},
			]
			
			\addplot[draw=none, fill=sidefillL, fill opacity=0.55, domain=-1:2/3]
			{max(1,3*abs(1-x))/3} \closedcycle;
			\addplot[draw=none, fill=sidefillR, fill opacity=0.55, domain=4/3:3]
			{max(1,3*abs(1-x))/3} \closedcycle;
			
			\addplot[draw=none, fill=flatfill, fill opacity=0.55, domain=2/3:4/3]
			{max(1,3*abs(1-x))/3} \closedcycle;
			
			\addplot[myblue, very thick] {max(1,3*abs(1-x))/3};

			\addplot[red, thick, dashed] coordinates {(2/3,0) (2/3,2)};
			\addplot[red, thick, dashed] coordinates {(4/3,0) (4/3,2)};

			\addplot[dotted, black!55, thin] coordinates {(-1,1/3) (3,1/3)};
			\node[font=\scriptsize, black!65, anchor=east] at (axis cs:-1.03,1/3) {$\tfrac13$};
			
			\fill[red] (axis cs:2/3,1/3) circle (2.6pt);
			\fill[red] (axis cs:4/3,1/3) circle (2.6pt);
			
			\node[
			circle,
			fill=myblue,
			draw=white,
			line width=0.4pt,
			inner sep=2pt
			]
			at (axis cs:1,1/3) {};
			
			\node[align=center, font=\scriptsize] at (axis cs:1,0.62)
			{\textbf{flat region}\\[-1pt] $\varphi\equiv\tfrac13$};
			
			\node[
			rotate=-45,
			font=\footnotesize\bfseries,
			text=myblue,
			text opacity=1,
			inner sep=1.5pt,
			rounded corners=1pt
			]
			at (axis cs:-0.38,1.57)
			{$\varphi(\mu)=1-\mu$};
			
			\node[
			rotate=45,
			font=\footnotesize\bfseries,
			text=myblue,
			text opacity=1,
			inner sep=1.5pt,
			rounded corners=1pt
			]
			at (axis cs:2.38,1.57)
			{$\varphi(\mu)=\mu-1$};
			
			\node[font=\scriptsize, red, below] at (axis cs:2/3,-0.06) {$\tfrac23$};
			\node[font=\scriptsize, red, below] at (axis cs:4/3,-0.06) {$\tfrac43$};

		\end{axis}
	\end{tikzpicture}
	\caption{
		The sharp bound $$\varphi(\mu)=\frac13\max\{1,\,3|1-\mu|\},$$		appearing in Theorem~\ref{Th-3}. The bound is constant on the interval 		$\mu\in[\frac23,\frac43]$ and increases linearly outside this interval.
		The transition points $\mu=\frac23$ and $\mu=\frac43$ separate the constant
		and linear extremal regimes.
	}
	\label{fig:bound}
\end{figure}
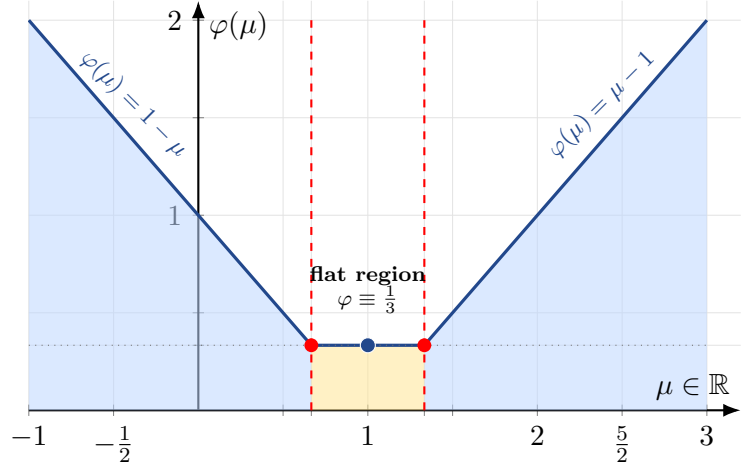

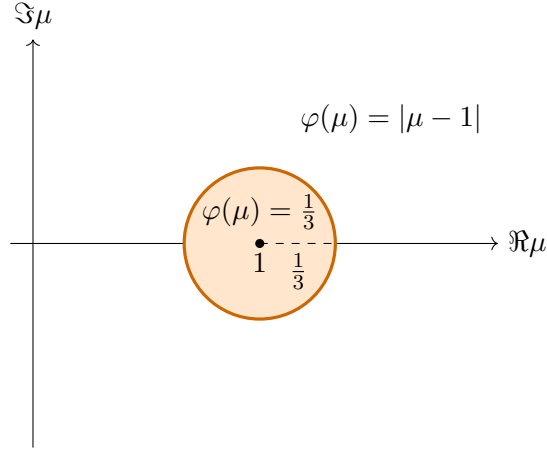
\begin{figure}[H]
	\centering
	\begin{tikzpicture}[scale=3]
		\draw[->] (-0.1,0) -- (2.05,0)
		node[right] {$\Re\mu$};
		\draw[->] (0,-0.9) -- (0,0.9)
		node[above] {$\Im\mu$};
		
		\fill[orange!20] (1,0) circle (0.3333);
		\draw[orange!80!black,very thick]
		(1,0) circle (0.3333);
		
		\fill (1,0) circle (0.02);
		\node[below] at (1,0) {$1$};
		
		\draw[dashed] (1,0) -- (1.3333,0);
		\node[below] at (1.17,0) {$\tfrac13$};
		
		\node[align=center] at (1,0.14)
		{$\varphi(\mu)=\tfrac13$};
		
		\node[align=center] at (1.58,0.55)
		{$\varphi(\mu)=|\mu-1|$};
	\end{tikzpicture}
	\caption{
		Geometry of the sharp bound
		\[
		\varphi(\mu)=\frac13\max\{1,3|1-\mu|\},
		\qquad \mu\in\mathbb C.
		\]
		The bound equals $\frac13$ in the disk
		$|\mu-1|\le\frac13$ and equals $|\mu-1|$ outside this disk.
	}
	\label{fig:convex-complex-bound}
\end{figure}

\begin{rem}
	Although Theorem~\ref{Th-3} is valid for every complex parameter $\mu$, the bound depends only on the distance $|\mu-1|$. Hence it is radially symmetric about $\mu=1$ in the complex plane. Figure~\ref{fig:bound} shows its restriction to the real axis, which completely determines the sharp bound.
\end{rem}
\begin{proof}[\bf Proof of Theorem \ref{Th-3}]
	By hypothesis, $F$ is a convex mapping. Using \eqref{C1} we have
	\begin{align*}
	\displaystyle
	\Re\left\lbrace l_z\left([DF(z)]^{-1}\left(D^2 F(z)(z^2)+DF(z) z\right)\right)\right\rbrace>0,\quad z\in\mathbb{B}\backslash\{0\},\quad l_z\in T(z)
	\end{align*}
	or equivalently
	\begin{align*}
	\displaystyle
		\Re\left\lbrace \frac{\|z\|}{l_z\left([DF(z)]^{-1}\left(D^2 F(z)(z^2)+DF(z) z\right)\right)}\right\rbrace>0, \quad z\in\mathbb{B}\backslash\{0\}.
	\end{align*}

	Let $z\in \mathbb{B}\backslash\{0\}$ be fixed and let $z_0=z/\|z\|$. We define the function $g:\mathbb{U}\to \mathbb{C}$ as follows
	\begin{align*}	
	\displaystyle
		g(\zeta)=\begin{cases}
			\frac{\zeta}{l_{z}\left([DF(\zeta z_0)]^{-1}\left(D^2 F(\zeta z_0)(\zeta^2z_0^2)+DF(\zeta z_0) \zeta z_0\right)\right)},& \zeta\neq 0,\\[8pt]
			1& \zeta=0.
		\end{cases}
	\end{align*}

Clearly $g\in\mathcal{H}(\mathbb{U})$ and $g(0)=1$.
Therefore for $\zeta\in\mathbb{U}\backslash\{0\}$, we have
\begin{align*}
\displaystyle
	g(\zeta)=&\frac{\zeta}{l_{z}\left([DF(\zeta z_0)]^{-1}\left(D^2 F(\zeta z_0)(\zeta^2z_0^2)+DF(\zeta z_0) \zeta z_0\right)\right)}\\=&\frac{||\zeta z_0||}{l_{\zeta z_0}\left([DF(\zeta z_0)]^{-1}\left(D^2 F(\zeta z_0)(\zeta^2z_0^2)+DF(\zeta z_0) \zeta z_0\right)\right)}.
\end{align*}	
Consequently, $\Re\{g(\zeta)\}>0$ for all $\zeta\in\mathbb{U}$.

From the series expansions of $g(\zeta)$ and $F(\zeta z_0)$, it follows that
\begin{align}\label{Th-3.1}
\displaystyle
	g(\zeta)=&1+g'(0)\zeta+\frac{g''(0)}{2}\zeta^2+\ldots,\\\label{Th-3.2}
\displaystyle
	F(\zeta z_0)=&\zeta z_0+\frac{1}{2!}D^2 F(0)((\zeta z_0)^2)+\frac{1}{3!}D^3 F(0)((\zeta z_0)^3)+\ldots.
\end{align}

Clearly, we have
\begin{align}\label{Th-3.3}
\displaystyle
D F(\zeta z_0)=&I+D^2 F(0)(\zeta z_0)+\frac{1}{2}D^3 F(0)((\zeta z_0)^2)+\ldots,\\\label{Th-3.4}
\displaystyle
	D^2 F(\zeta z_0)=&D^2F(0)+D^3F(0)(\zeta z_0)+\ldots
\end{align}

Since $F(z)=f(z)z$, the rank-one inverse formula used in
\eqref{Th-1.2b} gives
\begin{align}\label{Th-3.5}
\displaystyle
	[D F(z)]^{-1}\left(D^2F(z)(z^2)+DF(z) z\right)=\frac{D^2F(z)(z^2)+DF(z) z}{f(z)}\left\lbrack \frac{f(z)}{f(z)+Df(z)z}\right\rbrack.
\end{align}

Now using \eqref{Th-3.3} and \eqref{Th-3.4} to \eqref{Th-3.5}, we obtain
\begin{align}\label{Th-3.6}
\displaystyle
&[D F(\zeta z_0)]^{-1}\left(D^2 F(\zeta z_0)(\zeta^2z_0^2)+DF(\zeta z_0) \zeta z_0\right)\\=&
\left\lbrack \zeta z_0+2D^2 F(0)(\zeta^2 z_0^2)+\frac{3}{2} D^3 F(0)(\zeta^3 z_0^3)+\ldots\right\rbrack\left\lbrack \frac{1}{f(\zeta z_0)+Df(\zeta z_0)\zeta z_0}\right\rbrack.\nonumber	
\end{align}
Since $F(z)=zf(z)$, it follows that
\begin{align*}
\displaystyle
	\frac{D^3F(0)(z^3)}{3!}=\frac{D^2f(0)(z^2)}{2!}z,\quad \frac{D^2F(0)(z^2)}{2!}=Df(0)(z) z\quad \text{and so on}.
\end{align*}

Consequently
\begin{align}\label{Th-3.7}
\displaystyle
	\frac{l_{z}(D^3F(0)(z^3))}{3!}=\frac{D^2f(0)(z^2)}{2!}\|z\|,\quad \frac{l_z(D^2F(0)(z^2))}{2!}=Df(0)(z) \|z\|\quad \text{and so on}.
\end{align}

Now using \eqref{Th-3.7} to \eqref{Th-3.6}, we find that
\begin{align}\label{Th-3.8}
\displaystyle
g(\zeta)\left(1+4Df(0)(z_0)\zeta+\frac{9}{2} D^2f(0)(z_0^2)\zeta^2+\ldots\right)=f(\zeta z_0)+Df(\zeta z_0)\zeta z_0.
\end{align}

Using \eqref{Th-1.4}, \eqref{Th-1.5} and \eqref{Th-3.1} to \eqref{Th-3.8}, we get 
\begin{align}\label{Th-3.9}
\displaystyle
	&\left(1+h'(0)\zeta+\frac{h''(0)}{2}\zeta^2+\ldots\right)\left(1+4Df(0)(z_0)\zeta+\frac{9}{2} D^2f(0)(z_0^2)\zeta^2+\ldots\right)\\=&
	\left(1+Df(0)(z_0)\zeta+\frac{1}{2!}D^2 f(0)(z_0^2)\zeta^2+\ldots\right)+\left(Df(0)(z_0)\zeta+D^2 f(0)(z_0^2)\zeta^2+\ldots\right).\nonumber
\end{align}

Clearly from \eqref{Th-3.9}, we have
\begin{align}\label{Th-3.10}
\displaystyle
	h'(0)=-2Df(0)(z_0)\;\;\text{and}\;\;\frac{h''(0)}{2}=-3D^2f(0)(z_0^2)+8\left(Df(0)(z_0)\right)^2.
\end{align}

Since $z_0=z/||z||$, from \eqref{Th-3.7} and \eqref{Th-3.10}, we obtain 
\begin{align}\label{Th-3.11}
\displaystyle
	\frac{l_{z}(D^3F(0)(z^3))}{3!\|z\|^3}=&-\frac{1}{6}\left(\frac{h''(0)}{2}-2(h'(0))^2\right)\\\label{Th-3.12}
\displaystyle
	\frac{l_z(D^2F(0)(z^2))}{2!\|z\|^2}=&-\frac{1}{2}h'(0).
\end{align}
Now using \eqref{Th-3.11} and \eqref{Th-3.12}, we deduce that 
\begin{align*}
\displaystyle
	\left|\frac{l_{z}(D^3F(0)(z^3))}{3!\|z\|^3}-\mu\left(\frac{l_z(D^2F(0)(z^2))}{2!\|z\|^2}\right)^2\right|=&
	\left|\frac{1}{6}\frac{h''(0)}{2}+\frac{3\mu-4}{12}(h'(0))^2\right|\\=&
	\frac{1}{6}\left|\frac{h''(0)}{2}-\frac{1}{2}(h'(0))^2+\frac{3(\mu-1)}{2}(h'(0))^2\right|\\\leq&
	\frac{1}{6}\left(\left|\frac{h''(0)}{2}-\frac{1}{2}(h'(0))^2\right|+\frac{3|(1-\mu)|}{2}\left|h'(0)^2\right|\right).
\end{align*}
Therefore, using Lemma \ref{Lem2}, we get
\begin{align}\label{Th-3.13}
\displaystyle
	\left|\frac{l_{z}(D^3F(0)(z^3))}{3!\|z\|^3}-\mu\left(\frac{l_z(D^2F(0)(z^2))}{2!\|z\|^2}\right)^2\right|\leq 
	\frac{1}{6}\left(2-\frac{1}{2}|h'(0)|^2+\frac{3|1-\mu|}{2}|h'(0)|^2\right).
\end{align}
On the other hand, by Lemma \ref{Lem1}, we know that $|h'(0)|\in [0,2]$. Therefore maximizing the expression \eqref{Th-3.13} in $|h'(0)|$ over $[0,2]$ gives the result.\vspace{1.2mm}

To establish the sharpness of the bound, let us consider the function $F$ defined by
\begin{align*}
\displaystyle
	F(z)=zf(z)=\frac{f_1(l_u(z))}{l_u(z)}z,\quad z\in\mathbb{B},\quad \|u\|=1,
\end{align*}
where $f_1(l_u(z))=\frac{1}{2}\log \frac{1+l_u(z)}{1-l_u(z)}$. Obviously $f_1\in\mathcal{K}$. According to Corollary 2 \cite{Roper-Suffridge-1999}, $F\in\mathcal{K}_B(\mathbb{B})$. Note that
$f_1(l_u(z))=l_u(z)+\frac{1}{3}(l_u(z))^3+\frac{1}{5}(l_u(z))^5+\ldots.$
Also, we have
\begin{align*}
\displaystyle
F(z)=z+\frac{1}{3}(l_u(z))^2z+\frac{1}{5}(l_u(z))^4z+\ldots.
\end{align*}

It is clear that
\begin{align*}
\displaystyle
	\frac{D^2F(0)(z^2)}{2!}=0\quad \text{and}\quad \frac{D^3F(0)(z^3)}{3!}=\frac{1}{3}(l_u(z))^2z.
\end{align*}
Consequently
\begin{align*}
\displaystyle
	\frac{l_z(D^2F(0)(z^2))}{2!}=0\quad \text{and}\quad
	\frac{l_z(D^3F(0)(z^3))}{3!}=\frac{1}{3}(l_u(z))^2\|z\|.
\end{align*}
Taking $z=ru\;(0<r<1)$, we get
\begin{align*}
\displaystyle
	\frac{l_z(D^2F(0)(z^2))}{2!\|z\|^2}=0\quad \text{and}\quad
	\frac{l_z(D^3F(0)(z^3))}{3!\|z\|^3}=\frac{1}{3}.
\end{align*}
Therefore
\begin{align*}
\displaystyle
	\left|\frac{l_{z}(D^3F(0)(z^3))}{3!\|z\|^3}-\mu\left(\frac{l_z(D^2F(0)(z^2))}{2!\|z\|^2}\right)^2\right|=\frac{1}{3}.
\end{align*}
\end{proof}

\begin{cor}\label{Co-1}Let $f\in\mathcal{H}(\mathbb{B},\mathbb{C})$ with $f(0)=1$ and $F(z)=zf(z)\in\mathcal{K}_B(\mathbb{B})$. Then
	\begin{align*}
	\displaystyle
		\left|\frac{l_{z}(D^3F(0)(z^3))}{3!\|z\|^3}-\left(\frac{l_z(D^2F(0)(z^2))}{2!\|z\|^2}\right)^2\right|\leq 
		\frac{1}{3}\left(1-\left(\frac{l_z(D^2F(0)(z^2))}{2!\|z\|^2}\right)^2\right).
	\end{align*}
	The bound is sharp.
\end{cor}
\begin{proof} Using \eqref{Th-3.12} and \eqref{Th-3.13}, we obtain
\begin{align*}
\displaystyle
	\left|\frac{l_{z}(D^3F(0)(z^3))}{3!\|z\|^3}-\left(\frac{l_z(D^2F(0)(z^2))}{2!\|z\|^2}\right)^2\right|\leq&
	\frac{1}{6}\left(2-\frac{1}{2}|h'(0)|^2\right)\\\leq&
	\frac{1}{3}\left(1-\left(\frac{l_z(D^2F(0)(z^2))}{2!\|z\|^2}\right)^2\right).
\end{align*}
\end{proof}
\begin{rem} In the special case where $X = \mathbb{C}$ and $B = \mathbb{U}$, Corollary \ref{Co-1} reduces to Theorem C.
\end{rem}

\begin{cor}\label{co-2}Let $f\in\mathcal{H}(\mathbb{U}^n,\mathbb{C})$ with $f(0)=1$ and $F(z)=zf(z)\in\mathcal{K}_B(\mathbb{U}^n)$ and let $\mu$ be any complex number. Then
	\begin{align*}
		\displaystyle
		\left\|\frac{D^3F(0)(z^3)}{3!}-\mu\frac{1}{2!}D^2F(0)\left(z,\frac{D^2F(0)(z^2)}{2!}\right)\right\|_{\infty}\leq
		\frac{1}{3}\max\{1,3|1-\mu|\}||z||_{\infty}\nonumber
	\end{align*}
	for all $z\in\mathbb{U}^n$. The bound is sharp.
\end{cor}

\vspace{1.6mm}
Next, by removing the restrictive assumption $F(z) = f(z)z$, we generalize Theorem B to higher dimensions under weaker conditions than those in Theorem \ref{Th-3}.

\begin{theo}\label{Th-3a}
	Let $F$ be a locally biholomorphic mapping on $\mathbb{B}$, and $F$ satisfy the following assumptions:
	\begin{align*}
		\displaystyle
		\frac{D^{k+1}F(0)(z^{k+1})}{(k + 1)!} = H_{F,k}(z)z, \quad z \in \mathbb{B}, \; k = 1, 2, 3,
	\end{align*}
	where $H_{F,k}(z)$ is a homogeneous polynomial of degree $k$ with values in $\mathbb{C}$. Let $\mu$ be any complex number. If $F\in\mathcal{K}_B(\mathbb{B})$, then the inequality \eqref{Th3a} holds and the bound is sharp.
\end{theo}
\begin{proof}
By employing the same techniques as in the proof of Theorem \ref{Th-1a} and following the argument of Theorem \ref{Th-3}, the desired conclusion follows directly. We therefore omit the details.
\end{proof}
\vspace{5mm}

\noindent\textbf{Conflict of interest:} The authors declare that there is no conflict  of interest regarding the publication of this paper.\vspace{1.2mm}

\noindent {\bf Funding:} Not Applicable.\vspace{1.2mm}

\noindent\textbf{Data availability statement:}  Data sharing not applicable to this article as no datasets were generated or analysed during the current study.\vspace{1.2mm}

\noindent {\bf Authors' contributions:} All the authors have equal contributions in preparation of the manuscript.

\end{document}